\documentclass[11pt]{article}
\usepackage{epsf,textcomp}
 \usepackage[active]{srcltx}
\usepackage[usenames]{color}
\usepackage{psfrag,bbm,
graphicx, harvard,mathrsfs,
epsf, euscript,amsfonts,amsmath,latexsym,amssymb,latexsym,amssymb,
latexsym,amsthm}
\usepackage{changes}
\usepackage[toc,page]{appendix}

\usepackage{xcolor}
\usepackage{hyperref}

\newtheorem{theorem}{Theorem}

\newtheorem{remark}{Remark}

\usepackage{fullpage}

\newcommand{\rd}{{\rm d}}
\newcommand{\cE}{{\cal E}}

\newcommand{\cS}{{\cal S}}

\newcommand{\cV}{{\cal V}}

\newcommand{\cZ}{{\cal Z}}

\newcommand{\bE}{\mathbb E}

\newcommand{\bN}{{\mathbb N}}

\newcommand{\bP}{{\mathbb P}}

\newcommand{\bR}{{\mathbb R}}

\newcommand{\sB}{{\mathscr B}}

\newcommand{\sF}{{\mathscr F}}

\newcommand{\sT}{{\mathscr T}}

\begin{document}
\numberwithin{equation}{section}

\title{On optimal stopping problems with full information and rank-dependent reward}
\author{
A. Goldenshluger\thanks{Department of Statistics, University of Haifa, Haifa 3498838, Israel.  e-mail:
{\em goldensh@stat.haifa.ac.il}. }
\and Y. Malinovsky\thanks{Department of Mathematics and Statistics,
University of Maryland,
Baltimore County, Baltimore, MD 21250, USA.
e-mail: {\em yaakovm@umbc.edu}}
}
\date{}
\maketitle
\begin{abstract}
In this paper we consider optimal stopping problems with full information and rank--dependent rewards. Our focus is on the following two problems: selecting one of the $r$~best observations (Gusein--Zade's problem with full information) and
minimizing the expected rank (Robbins' problem). In both settings we provide
full characterization of the optimal stopping rules and corresponding
optimal values. We also present numerical experiments in which the derived dynamic programming recursions are implemented, and   approximations
to optimal values are computed.
\end{abstract}
\vspace*{1em}
\par
\vspace*{1em}

\section{Introduction}
Let $X_1, \ldots, X_n$ be a sequence of independent random variables 
uniformly distributed on $[0,1]$ and defined on probability space 
$(\Omega, \sF, \bP)$. Let $A_t$ and $R_t$ be the absolute and 
relative  
ranks of observation $X_t$ respectively:
\[
 A_t:= \sum_{j=1}^n I(X_t\leq X_j),\;\;\;R_t:=\sum_{j=1}^t I(X_t\leq X_j),\;\;t=1, \ldots, n.
\]
Let $\sF_t=\sigma(X_1, \ldots, X_t)$ be the $\sigma$-field 
generated by $(X_1, \ldots, X_t)$, and $\{\sF_t, 1\leq t\leq n\}$ be the corresponding filtration. {\em A stopping time} $\tau$ with respect to 
$\{\sF_t, 1\leq t\leq n\}$ is a random variable on $(\Omega, \sF, \bP)$ such that $\bP\big\{\tau\in \{1, \ldots, n\}\big\}=1$, and $\{\tau=t\}\in \sF_t$, $1\leq t\leq n$. The set of all possible stopping times with respect to
filtration $\{\sF_t, 1\leq t\leq n\}$ is denoted $\sT$.
\par 
Let $q:\{1, \ldots, n\}\to \bR$ be a reward function. The average reward of 
a stopping time $\tau\in \sT$ is defined by 
\[
 V_n(q;\tau):= \bE q(A_\tau).
\]
The objective is to find the optimal stopping rule $\tau_*\in \sT$ such that 
\[
 V_n(q; \tau_*)= V_n^*(q):=\sup_{\tau\in \sT} \bE q(A_\tau)= \bE q(A_{\tau_*}).
\]
The stated problem is the one of optimal stopping with full information and rank--dependent reward~$q$.
\par 
Although the main idea of this paper is applicable
to
general reward functions $q$, our focus will be on the following two important specific settings.
\begin{itemize}
 \item[(a)] {\em Gusein--Zade's stopping problem with full information.} 
 In this setting 
 $q(a)= q_{\rm GZ}(a):=I(a\leq r)$, $1\leq r\leq n$, and the problem is to maximize 
 the probability of selecting one of the $r$ largest observations using 
 observations $X_1, \ldots, X_n$:
 \[
  V_n^*(q_{\rm GZ}):=\sup_{\tau\in \sT} \bP\{ A_\tau\leq r\}.
 \]
\item[(b)] {\em Robbins' optimal stopping problem.} In this celebrated problem the objective is to find the stopping time $\tau\in \sT$ 
that minimizes the expected absolute rank, i.e., to find $\tau_*\in \sT$
such that 
\[
 W_n^*(q_{\rm R}):=\inf_{\tau\in \sT} \bE A_{\tau}= \bE A_{\tau_*}.
\]
To treat the problems in the unified way, we put $q_{\rm R}(a)= n-a$ and state the problem as the problem of finding the stopping rule 
$\tau_*\in \sT$ such that  
\[
 V_n^*(q_{\rm R}):=\sup_{\tau\in \sT} V_n(q_{\rm R}; \tau)= n - \bE A_{\tau_*}.
\]
Then the optimal value of Robbins' problem is $W_n^*:=W_n^*(q_{\rm R})= n- V_n^*(q_{\rm R})$.
\end{itemize}
\par
The problem~(a) was introduced and solved in \citeasnoun{GuZa} in the 
{\em no--information} setting, when only relative ranks $\{R_t\}$ of $X_1, \ldots, X_n$ are observed sequentially. The setting when $X_1, \ldots, X_n$ are observed
and $r=1$
is the classical secretary problem with {\em full information}. This problem has been studied extensively.
It has been introduced and heuristically solved by \citeasnoun{GiMo}. For rigorous treatments we refer to
\citeasnoun{Bojdecki78} and \citeasnoun{Samuels82}.
It is worth mentioning that the solution in
\citeasnoun{Bojdecki78} is based on application of the
general optimal stopping theory to
the embedded Markov chain $(s_k, X_{s_k})_{k\geq 1}$, where $s_k$ are the moments
at which the records of sequence $(X_t)_{t\geq 1}$ are observed. \citeasnoun{Gnedin96} also discusses the best--choice problem with full information and develops the planar Poisson process (PPP) framework for its analysis.
To the best of our knowledge, characterization of the
optimal stopping rule in the full information Gusein--Zade's problem
with general $r\geq 2$ is not available in the literature.
\par
Robbins' problem  has attracted much more attention. The history of this celebrated  problem is nicely presented in
\citeasnoun{Bru05}; see also \citeasnoun{Bru2024}. Exact solutions to Robbins' problem
are available only for $n\leq 4$; see \citeasnoun{Dendievel}.
For general values of $n$ the problem remains open to this day, and
the main challenge is to find the optimal stopping rule and
the corresponding limiting optimal value.
\par 
There have been many efforts to understand the form of the optimal strategy and
to derive upper and lower bounds on the limiting value $W^*:=\lim_{n\to\infty} W_n^*$. The following results have been proved in
\citeasnoun{Bru93} and \citeasnoun{Ass96}.
It was shown that the sequence $(W_n^*)$ is monotone increasing in $n$. The
boundedness and, therefore, existence
of the limit $W^*=\lim_{n\to\infty} W_n^*$ follow from the results of
\citeasnoun{chow}, where  minimizing the expected
absolute rank with no--information is considered.
The optimal rule in $n$--period problem 
depends on the whole history, $X_1, \ldots, X_n$, and it was shown that 
$W_n^*\leq  7/3$ for all $n$.  This upper bound is achieved 
by a memoryless threshold rule
that stops the first time $t$ when 
the observed random variable $X_t$ exceeds a threshold.
\citeasnoun{Ass96} showed that a value 2.3318 
can be obtained with thresholds of a simple form. 
In addition, a sequence of lower bounds on $W^*$ was derived in \citeasnoun{Bru93} 
that converges to a limiting value which is not smaller than 1.908.
\citeasnoun{BruFer96} study structural 
properties of the optimal stopping rule. It is shown there that
the optimal stopping rule depends on the whole history, and that 
the optimal threshold function are stepwise monotone. 
\par
Another strand of research dealing with Robbins' problem considers Poisson embeddings; see \citeasnoun{Gnedin07}, \citeasnoun{BruSwan} and \citeasnoun{MeierSogner}.  
These papers demonstrate how upper bounds on the limiting optimal value in
the original problem can be derived from continuous time Poisson embeddings. 
\citeasnoun{Gnedin07} shows that dependence of the whole history persists in asymptotics as $n\to\infty$, so that memoryless threshold rules cannot be asymptotically optimal. The stopping strategy proposed in \citeasnoun{MeierSogner} 
combines a rank--dependent rule for the best--choice problem when $t\leq \alpha$
for some $\alpha$  with 
a threshold rule when $t>\alpha$.  It is claimed that such a strategy results at the limiting optimal value 2.33044.
Recent paper \citeasnoun{Brice2026} considers Robbins' problem as 
an infinite Markov Decision Process (MDP), where 
the states are the observations obtained so far, and the actions are the decisions to stop or not to stop. If the decision ``not to stop'' is made, then a new observation 
from $[0,1]$ is added to the available data. In such a model the state space in uncountable. The paper proposes a discretized model in which  the state space 
of the corresponding  MDP is finite. This interesting approach  enabled to derive approximations to the optimal values 
for $5\leq n\leq 100$. 
\par 
The main contributions of the present paper are the following. We provide complete characterizations  of the optimal stopping rules and optimal values in
the full information Gusein--Zade's stopping problem and in Robbins' problem.
We show that the optimal rules have the following general form
\begin{equation}\label{eq:stop-rule-form}
 \tau_*=\min\Big\{t\in\{1, \ldots,n\}: g_{t}(R_t, X_t) \geq C_{n-t}(Z_t)\Big\}.
\end{equation}
The functions $g_t(\cdot, \cdot)$ in \eqref{eq:stop-rule-form} 
are fully determined by the 
reward function $q$, while  
the threshold functions $C_t(\cdot)$ are defined  through  
explicit dynamic programming recursions; here  $Z_t$ is the vector given by 
$Z_t:=(X_{1:t}, \ldots, X_{1:r})$ for Gusein--Zade's problem and 
$Z_t:=(X_{1:t}, \ldots, X_{t:t})$ 
for Robbins' problem with
$X_{1:t}\geq X_{2:t}\geq \cdots\geq X_{t:t}$, $t=1, \ldots, t$ being the order statistics of the sample.
The proofs of main results of this paper, Theorems~\ref{th:GZ} and~\ref{th:Robbins}, are based on construction of a Markov representation
for the adapted reward and on application 
of the general theorem for stopping Markov sequences
\cite{ChowRobSieg}. We conducted a small numerical study in which 
the derived recursions were implemented, and  approximate optimal values were computed for a range of different values of $n$ (and $r$ for Gusein--Zade's problem). 
\par 
\citeasnoun{Bru05} discusses what
 should really  be considered as
a complete solution to the Robbins' problem and
poses the following three questions.
\begin{itemize}
\item[(i)] What is the limit $W^*=\lim_{n\to\infty} W_n^*$?
\item[(ii)] What is the form of the optimal strategy?
\item[(iii)] For any $n\in \bN$ and $\epsilon>0$, how can we determine
an $n(\epsilon)$--memory algorithm to compute an $\epsilon$--optimal strategy,
i.e. a strategy that achieves the value of at least $W_n^* -\epsilon$ by remembering, at each stage, no more than a fixed number $n(\epsilon)$ of observed
points?
\end{itemize}
Section~\ref{sec:Robbins} of this paper provides a complete answer to question~(ii).
\par 
The rest of this paper is organized as follows. Section~\ref{sec:reward}
presents a formula for the adapted reward in the full information stopping problem 
with general reward function $q$. 
Section~\ref{sec:GZ} deals with Gusein--Zade's problem with full information:
the main result is given in Theorem~\ref{th:GZ}, and 
Subsections~\ref{sec:Discussion} and \ref{sec:Numerics-GZ} present discussion and 
numerical experiments, respectively. In Section~\ref{sec:Robbins} we consider Robbins' problem. Theorem~\ref{th:Robbins} provides a characterization of the optimal rule and the optimal value, and Subsections~\ref{sec:Discussion-Robbins} and \ref{sec:Numerics-Robbins} present discussion and numerical results related to Robbins' problem.
Section~\ref{sec:conclusion} brings some concluding remarks. 

\section{Adapted reward}\label{sec:reward}
For $1\leq t\leq n$ we first compute 
 $\bE \{q(A_t)|\sF_t\}$:
 \begin{align}
  Y_t:=\bE \{q(A_t)| \sF_t\} &= \bE\Big\{ 
  q\Big(R_t+\sum_{i=t+1}^n I(X_i \geq X_t)\Big) \,\Big|\, \sF_t\Big\}
  \nonumber
  \\
  &= \sum_{j=0}^{n-t} q(R_t+j) \binom{n-t}{j}(1-X_t)^j X_t^{n-t-j},
 \label{eq:adapted-reward}
 \end{align}
where the last line follows from the fact that 
\[
 \sum_{i=t+1}^n I(X_i \geq X_t)\,\Big|\, X_1, \ldots, X_t\;\sim\;
 {\rm Binomial}(n-t, 1-X_t).
\]
Because for every $\tau\in \sT$
\[
 \bE q(A_\tau)=\sum_{t=1}^n \bE \big[ I(\tau=t) \bE\{q(A_t)|\sF_t\}\big]= \bE Y_\tau,
\]
the original problem is equivalent to the problem of 
stopping the sequence of random variables
$Y_1, \ldots, Y_n$:
\[
 V_n^*(q)= \bE q(A_{\tau_*})= \bE Y_{\tau_*}.
\]
\par 
For $1\leq t\leq n$
define functions $g_t: \{1, \ldots, t\}\times [0,1]\to \bR_+$
by 
\begin{equation}\label{eq:g-t}
 g_t(k, x):= \sum_{j=0}^{n-t} q(k+j) \binom{n-t}{j}(1-x)^j x^{n-t-j}.
\end{equation}
With this notation 
\begin{equation}
\label{eq:Y}
Y_t= g_t(R_t, X_t),\,\, 1\leq t\leq n.
\end{equation}
We call the sequence $(Y_t)$ {\em the adapted reward}. 

\section{Gusein--Zade's problem with full information}\label{sec:GZ}
In Gusein--Zade's stopping problem $q(a)=q_{\rm GZ}(a)= I(a\leq r)$, and  
\begin{equation} \label{eq:g-GZ}
 g_t(k, x)
  = I(k\leq r)\sum_{j=0}^{(r-k)\wedge (n-t)} \binom{n-t}{j}(1-x)^j x^{n-t-j},\;\;\;k=1, \ldots, n.
\end{equation}
According to 
\eqref{eq:Y}, the adapted reward is given by
\begin{align*}
 Y_t &=\bP\{A_t\leq r|\sF_t\} = g_t(R_t, X_t).
\end{align*} 
\par
At time instance $t$
let $Z_t$ be the vector of $r$ largest order statistics 
$X_{1:t}\geq X_{2:t} \geq \cdots \geq X_{r:t}$
of $X_1, \ldots, X_t$ with the convention that $X_{k:t}=0$ for $k>t$, i.e.,
\begin{eqnarray} 
Z_t &:=& (X_{1:t}, \ldots, X_{r:t}), \;\;\hspace{15mm}\hbox{ if   }\; t\geq r,
\label{eq:Z-t}
\\
Z_t &:=& (X_{1:t},\ldots, X_{t:t}, \underbrace{0,\ldots, 0}_{r-t}), 
             \;\;\;\;\hbox{if }\; 1\leq t<r.
             \nonumber
\end{eqnarray}
For definiteness, we also put $Z_0=0_r$, where $0_r$ is the $r$--tuple of zeros, $0_r:=(\overbrace{0,\ldots, 0}^r)$. 
Vectors $Z_t$, $t=1, \ldots, n$  take values in the set 
\[
 \cE_r:=\{ z=(z_1, \ldots, z_r): 1\geq z_1\geq  z_2\geq \cdots \geq z_r\geq 0\}.
\]
By convention, we put also $z_0=1$.
\par 
Define the map $T: \cE_r\times [0,1] \to \cE_r$ by the following formula:
\[
 T(z,u):= \left\{\begin{array}{ll}
                  T_j(z,u) , & \hbox{if   } z_j<u\leq z_{j-1},\;\;j=1, \ldots, r,\\[2mm]
                  z, & \hbox{if   }u \leq z_r,
                 \end{array}
                 \right.
\]
where 
\begin{equation}\label{eq:T-j}
 T_j(z,u) :=(z_1,\ldots, z_{j-1}, u, z_{j}, \ldots, z_{r-1}),\;\;\;j=1,\ldots, r.
\end{equation}
In words, $T(z,u)$ is the vector comprised of $r$ ordered
largest values of $(z,u)\in \cE_r\times [0,1]$. Vector 
$T_j(z,u)$ is obtained when $u$ has relative rank $j=1, \ldots,r$ 
with respect to the components of vector~$z$.
\subsection{Optimal stopping rule and optimal value}
The following statement provides a characterization of the optimal stopping rule and the corresponding optimal value.
\begin{theorem}\label{th:GZ}
Let $C_m(z)$, $m=0, \ldots, n$ be functions on $\cE_r$ defined recursively
by
\begin{align}
C_0(z) & = 0, 
\label{eq:C-0}
\\
 C_m(z) &= z_r C_{m-1}(z)+ \sum_{j=1}^r\int_{z_j}^{z_{j-1}}
 \max\big\{g_{n-m+1}(j, u), C_{m-1}(T_j(z,u))\big\}\rd u,\,\,m\geq 1,
 \label{eq:C-recursion}
\end{align}
where $g_m(j,u)$ and $T_j(\cdot, \cdot)$ are given in 
\eqref{eq:g-GZ} and \eqref{eq:T-j} respectively. Then  
 the optimal stopping rule~is 
\begin{equation}\label{eq:GZ-stop}
 \tau_*:=\inf \Big\{1\leq t\leq n: g_{t}(R_t, X_t)\geq C_{n-t}(Z_t)\Big\},
\end{equation}
and the corresponding optimal value is 
\[
 V_n^*(q_{\rm GZ})= C_n(0).
\]
\end{theorem}

\begin{remark}
 Because $g_n(k, x)=I(k\leq r)$ and $C_0(z)=0$, 
 the inequality in \eqref{eq:GZ-stop} is always fulfilled for $t=n$. 
Thus, the stopping set in the definition of $\tau_*$ is
non--empty.
 \end{remark}

 \begin{remark}
  In the extended form the recursion is written as follows
  \begin{align*}
    C_m(z_1, \ldots, z_r) &= z_r C_{m-1}(z_1, \ldots, z_r)
   \\
   &\;\;\;+\;\sum_{j=1}^r \int_{z_j}^{z_{j-1}} \max\big\{ g_{n-m+1}(j, u), 
   C_{m-1}(z_1, \ldots, z_{j-1}, u, z_j, \ldots, z_{r-1})\big\}\rd u,
  \end{align*}
where $g_m(j, u)$ is given in \eqref{eq:g-GZ}.
 \end{remark}
\begin{proof}
{\bf 1$^0$.} First we demonstrate that the adapted reward admits 
a Markovian representation. 
Let $\cS_r:=\cE_r\times [0,1]$, and define random process
\begin{equation}
\label{eq:S-t}
 S_t := (Z_t, X_t)= (T(Z_{t-1}, X_t), X_t),\;\;\;t=1,2, \ldots
\end{equation}
on state space $(\cS_r, \sB)$,
where  $\sB$ is the Borel $\sigma$--filed on $\cS_r$. Observe that
 $(S_t)_{t\geq 0}$ is a time--homogeneous Markov chain on 
 $(\cS_r, \sB)$: for every   $B\in \sB$ and $t\in \bN_0$ one has 
\begin{equation}
 \bP(S_{t+1}\in B\,|\,S_1, \ldots, S_t) = K(S_t, B),
\end{equation}
where 
\begin{equation}\label{eq:K}
 K(s, B)= K((z,x), B):=\int_0^1 I\{(T(z,u), u)\in B\} \rd u.
\end{equation}
Indeed,  for any bounded measurable function $f$
 \[
  f(S_{t+1})= f(T(Z_t, X_{t+1}), X_{t+1}).
 \]
 Since $X_{t+1}$ is independent of $\sF_{t}$, 
\begin{align}\label{eq:1}
 \bE \{f(S_{t+1})|\sF_t\}= \bE\{f(T(Z_t, X_{t+1}), X_{t+1})|\sF_t\}
 =\int_0^1 f(T(Z_t, u), u)\rd u.
\end{align}
The right hand side depends on $(X_1, \ldots, X_t)$  
only through $Z_t$, and hence only through $S_t$. 
Since $S_t$ is $\sF_t$--measurable, in view of  \eqref{eq:1}, 
\[ 
\bE\{ f(S_{t+1})| S_1, \ldots, S_{t}\}=
\bE\Big[\bE\{ f(S_{t+1})| \sF_t\}\,| S_1, \ldots, S_{t}\Big]
= \bE\{ f(S_{t+1})|S_t\}. 
\]
Thus, $\{S_t\}$ is  a Markov chain, as claimed.
\par 
The adapted reward $Y_t$ 
can be represented as a function of the state of the process
$(S_t)_{t\geq 0}$. For $(z,x)=(z_1, \ldots, z_r, x)\in \cS_r$ and 
$t=1, \ldots, n$ we define
\[
 G_t(z,x):=\left\{\begin{array}{ll}
g_{t}(k, x), & \hbox{if  } x=z_k,\;\;k=1, \ldots, r,\\[2mm]
0, & \hbox{if  } x<z_r,
                  \end{array}
\right.
\]
where $g_t(k, x)$ is given in \eqref{eq:g-GZ}. With this notation we have that 
\[
 Y_t= G_{t}(S_t),\;\;\ S_t := (Z_t, X_t)= (T(Z_{t-1}, X_t), X_t),\;\;\;t=1,2, \ldots.
\]
In words, $Y_t=G_{t}(S_t)=g_t(R_t, X_t)$ is the reward of stopping
after observing $X_t\sim U(0,1)$. The reward depends on 
the relative rank of $X_t$ among $r$ largest order statistics observed
until time~$t$; if the relative rank of $X_t$ is greater than $r$, 
the reward is zero.  
With this Markovian representation we can apply the general theory of optimal stopping to the Markov sequence $(Y_t)_{t\geq 1}$; see \citeasnoun{ChowRobSieg}.

\par\medskip 
{\bf 2$^0$.}
Let $V_m(z, x)$ be the optimal success probability from  state
$(z,x)\in \cS_r$ when we examine observation $X_t$ for stopping and 
there are $m=n-t$ steps to go.
In view of the Markov property of $\{S_t\}$, 
the dynamic programming yields 
\begin{align*}
 V_m(z,x)&= 
 \max\big\{G_{n-m}(z,x),\; \bE \big[V_{m-1}(S_{t+1})|S_t=(z,x)\big]\big\},\;\;\;1\leq m\leq n-1,
 \\
 V_0(z,x)&= G_n(z,x).
\end{align*}
Because the transition kernel in \eqref{eq:K} does not depend on $x$, the conditional expectation on the right hand side of the above formula~is a function of $z$ 
only. Let us introduce the following notation    
\[
C_m(z):= \bE \big[V_{m-1}(S_{t+1})|S_t=(z,x)\big] = \bE \big[V_{m-1}(S_{t+1})|Z_t=z\big].
\]
In words, 
$C_m(z)$ is the expected conditional reward from continuation of sampling,
given  that $S_t=(z,x)$ and there are $m$ steps to go.
The dynamic programming recursion takes the form
\begin{equation}\label{eq:Vm=}
 V_m(z,x)= 
 \max\big\{G_{n-m}(z,x),\; C_m(z)\big\},\;\;\;\forall m\geq 1.
\end{equation}
Furthermore, by definition of $T(z,u)$ in \eqref{eq:T-j} and by \eqref{eq:S-t}, 
\begin{align*}
 C_m(z) & =
 \bE \big[V_{m-1}(S_{t+1})|Z_t=z\big]=
 \int_0^1 V_{m-1}(T(z, u), u) \rd u
 \\
 &= \int_0^{z_r} V_{m-1}(z,u) \rd u + \sum_{j=1}^r\int_{z_j}^{z_{j-1}}
 V_{m-1}(T_j(z,u), u) \rd u.
\end{align*}
If  $0< u < z_r$
then the relative rank of $u$ is not among the top $r$ ranks; 
hence $G_{n-m+1}(z,u)=0$, and 
in view
of~\eqref{eq:Vm=},
\[
 V_{m-1}(z,u)= \max\{0, C_{m-1}(z)\}= C_{m-1}(z),\;\;0<u<z_r.
\]
This implies that 
\[
 \int_0^{z_r} V_{m-1}(z,u) \rd u= \int_0^{z_r} C_{m-1}(z) \rd u= z_r C_{m-1}(z).
\]
Furthermore, 
for $z_j< u< z_{j-1}$ we have that 
$G_{n-m+1}(T_j(z,u), u)= {g_{n-m+1}(j, u)}$;
hence 
\[ 
V_{m-1}(T_j(z,u))=\max\{g_{n-m+1}(j,u), C_{m-1}(T_j(z,u))\},\;\;\;
z_j< u< z_{j-1}.
\]
Combining these relationships we obtain \eqref{eq:C-recursion}.
At the last observation when $m=0$ we have  $C_0(z)=C_0(z_1,\ldots, z_r)=0$.
This completes the proof.
\end{proof}

\subsection{Discussion}\label{sec:Discussion}
In this subsection we 
specialize  recursive formula   \eqref{eq:C-recursion}
for particular cases of $r=1$ and $r=2$.
We also discuss implementation of 
the recursion for computing threshold functions $C_m(z)$ and
optimal values $C_n(0)$
numerically for a range of $r$ and $n$.

\subsubsection{Classical secretary problem with full information} 
For $r=1$, appealing to \eqref{eq:g-GZ} and \eqref{eq:Z-t}, we have 
\[
Z_t=\max_{1\leq j\leq t}X_j,\;\;\;\;\;\;
g_{t}(k,x)= I(k=1) x^{n-t}=\left\{\begin{array}{ll} x^{n-t}, & k=1,\\
  0, & k>0.                 
                  \end{array}
 \right.
\]
The recursion \eqref{eq:C-recursion} becomes
\begin{align}
 C_m(z) &= zC_{m-1}(z)+
 \int_z^1\max\{u^{m-1},C_{m-1}(u\vee z)\}\rd u,
 \nonumber
 \\
 &= zC_{m-1}(z)+
 \int_z^1\max\{u^{m-1},C_{m-1}(u)\}\rd u,
 \qquad C_0(z)=0.
 \label{eq:r1-recursion}
\end{align}
In particular, 
\[
 C_1(z)=1-z,\;\;\;\;C_2(z)=\left\{\begin{array}{ll}
-\tfrac{1}{2}z^2+\tfrac{3}{4}, & z\leq \tfrac{1}{2},\\[2mm]
-\tfrac{3}{2}z^2 +z +\tfrac{1}{2}, & z>\tfrac{1}{2}.
                              \end{array}
\right.
\]
Figure~\ref{fig:Cr=1} displays functions $C_m(z)$ for a range of
different values of $m$.
\par 
The optimal stopping rule is
\[
 \tau_*= \min \Big\{1\leq t\leq n: R_t=1 \hbox{  and  } X_t^{n-t}\geq C_{n-t}(Z_t)\Big\}.
\]
The recursive formula \eqref{eq:r1-recursion} coincides with the one given in \citeasnoun{Samuels82}. The optimal values of $C_n(0)$ for 
different
values of $n$ are tabulated in Table~8 of \citeasnoun{GiMo}.

\begin{figure}
\begin{center}
 \includegraphics[scale=0.35]{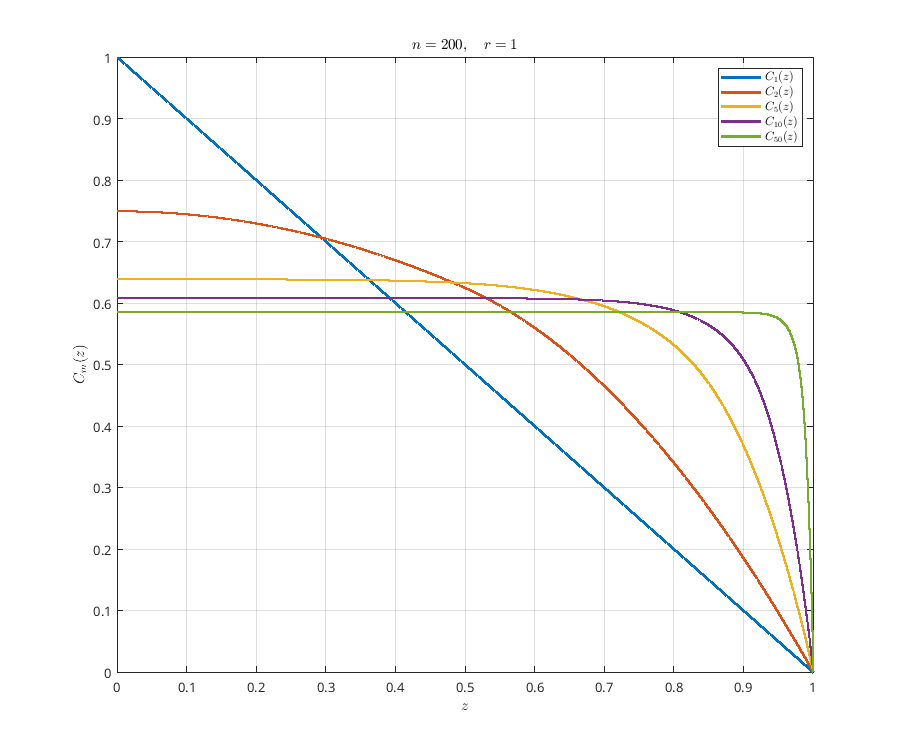}
\end{center}
\caption{Functions $C_m(z)$ given by \eqref{eq:r1-recursion} for 
 $m=1, 2,5, 10, 50$.\label{fig:Cr=1}}
 \end{figure}
\subsubsection{Selecting one of the two largest observations with full information}
For $r=2$ we have
$Z_t=(X_{1:t},X_{2:t})$,
where $X_{1:t}\geq X_{2:t}$ and, by convention,
$X_{2:1}=0$. Furthermore, for $1\leq t\leq n-1$ using \eqref{eq:g-GZ} we have
\begin{align*}
 g_{t}(1, x)
 &=x^{n-t}+(n-t)(1-x)x^{n-t-1}, 
 \\
 g_{t}(2, x)
 &=x^{n-t}, 
\end{align*}
and $g_n(k, x)=1$,
$g_{0}(k,x)=1$ for $k=1,2$.
Writing $z=(z_1,z_2)$, the recursion
\eqref{eq:C-recursion} becomes
\begin{align}
 C_m(z_1,z_2)
 =z_2C_{m-1}(z_1,z_2) +\int_{z_1}^1
 \max\big\{g_{n-m+1}(1, u),
 C_{m-1}(u,z_1)\big\}\rd u \nonumber\\
 +\int_{z_2}^{z_1}
 \max\big\{g_{n-m+1}(2, u),
 C_{m-1}(z_1,u)\big\}\rd u,
 \label{eq:r2-recursion}
\end{align}
and $C_0(z_1, z_2)=0$.
The optimal stopping rule \eqref{eq:GZ-stop} takes the form 
\begin{align*}
&\tau_*=\min\Big\{1\leq t\leq n: I(E_t)=1\Big\},
\\
&E_t:=
\Big\{R_t=1,\; X_t^{n-t}+(n-t)(1-X_t)X_t^{n-t-1} \geq C_{n-t}(Z_t)\Big\}\cup 
\Big\{R_t=2,\; X_t^{n-t}\geq C_{n-t}(Z_t)\Big\}.
\end{align*}

\subsection{Numerical experiments}\label{sec:Numerics-GZ}

\paragraph{Algorithm.}
In our implementation the one--dimensional integrals in \eqref{eq:C-recursion}
are approximated  numerically
using the trapezoid rule on the points
of the non--uniform grid 
\[
x_i=1-\Big(1-\frac{i}{K}\Big)^3,\;\;i=0, 1, \ldots, K,
\]
where $K$ is the grid size. The points of the grid are chosen to be denser near one because in this region the binomial tail reward varies rapidly 
for large $m$.
\par
Define the set of grid--index tuples by
\[
 \cV_r:=\left\{ v=(v_1, \ldots, v_r)\in \bN^r: K+1\geq v_1\geq \ldots\geq v_r\geq  1\right\}.
\]
With each $r$--tuple $v\in \cV_r$ of integers
we associate the
{\em grid state}
$
x(v):=(x_{v_1}, \ldots, x_{v_r})\in \cE_r$. The zero state has the index tuple
$(1, \ldots, 1)$.
The total number of grid states is
\[
|\cV_r|=\binom{K+r}{r},
\]
which is the number of multisets of size $r$ in
$K+1$ grid indices.
The algorithm approximates values of functions $C_m$ at the grid points,
$C_m(x(v))$, $v\in \cV_r$, and we denote the corresponding approximations by
$\hat{C}_m (x(v))$.
\par
The approximations $\hat{C}_m(x(v))$
are constructed as follows.
Fix $v\in \cV_r$, $m\geq 1$ and $j\in \{1, \ldots, r\}$. On
every interval $[x_{v_j}, x_{v_{j-1}}]$
we approximate the corresponding integrals in \eqref{eq:C-recursion}
in the following way:
\begin{multline}
 \int_{x_{v_j}}^{x_{v_{j-1}}} \max\big\{g_{n-m+1}(j,u), C_{m-1}(x_{v_1},\ldots, x_{v{j-1}}, u, x_{v_j}, \ldots x_{v_{r-1}})\big\} \rd u
 \\
\;\; \approx \frac{1}{2} \sum_{k=v_j}^{v_{j-1}} (x_{k+1}-x_k)
 \Big[  \max\big\{g_{n-m+1}(j,x_k), \hat{C}_{m-1}(x_{v_1},\ldots, x_{v_{j-1}}, x_k, x_{v_j}, \ldots x_{v_{r-1}})\big\}
 \\
+ \max\big\{g_{n-m+1}(j,x_{k+1}), \hat{C}_{m-1}(x_{v_1},\ldots, x_{v_{j-1}}, x_{k+1}, x_{v_j}, \ldots x_{v_{r-1}})\big\}\Big].
\label{eq:C-approximation}
 \end{multline}
 By convention, the empty sum is defined to be zero. Then the complete discrete
 recursion for $m=1, \ldots, n$ is
 \begin{multline*}
  \hat{C}_m(x_{v_1},\ldots, x_{v_r}) = x_{v_r}\hat{C}_{m-1}(x_{v_1},\ldots, x_{v_r})
\\
+ \frac{1}{2} \sum_{j=1}^r \sum_{k=v_j}^{v_{j-1}} (x_{k+1}-x_k)
 \Big[  \max\big\{g_{n-m+1}(j,x_k), \hat{C}_{m-1}(x_{v_1},\ldots, x_{v_{j-1}}, x_k, x_{v_j}, \ldots x_{v_{r-1}})\big\}
 \\
 + \max\big\{g_{n-m+1}(j,x_{k+1}), \hat{C}_{m-1}(x_{v_1},\ldots, x_{v_{j-1}}, x_{k+1}, x_{v_j}, \ldots x_{v_{r-1}})\big\}\Big]
  \end{multline*}
with the initialization $\hat{C}_0(x_{v_1}, \ldots, x_{v_r})=0$ for all $v\in \cV_r$.
The output is $\hat{C}_n(x_0, \ldots, x_0)=\hat{C}_n(0,\ldots, 0)$.
\par
The stopping rewards $g_{n-m+1}(j, x)$ for $r-j\geq m-1$ are equal~1; otherwise, they are evaluated using the regularized incomplete beta function [see formula 6.6.4
in \citeasnoun{Abramowitz}]:
\[
 g_{n-m+1}(j, x)= I_x(m-r+j-1, r-1),\;\;\;\;
 I_x(a,b)=\frac{\int_0^x t^{a-1}(1-t)^{b-1}\rd t}{\int_0^1 t^{a-1}(1-t)^{b-1}\rd t}.
\]
\par
To compute $\hat{C}_m(x_{v_1}, \ldots, x_{v_{r}})$ for every state $x(v)=(x_{v_1}, \ldots, x_{v_r})$ the algorithm considers {\em prefix}
$(x_{v_1},\ldots, x_{v_{r-1}})$. Fixing the prefix,  we have
a groups of states that are determined by sorting the tuples
$(x_{v_1},\ldots, x_{v_{r-1}}, x_i)$, $i=1, \ldots, v_{r-1}$.
Thus, for every $m$ the algorithm processes $\binom{K+r-1}{r-1}$ prefixes,
each with $K+1$ resulting states ($K$ intervals). Thus, for each $m$,
the algorithm sorts $\binom{K+r-1}{r-1}(K+1)$ tuples of length $r$ and computes
the corresponding values of function $\hat{C}_m$ at these points. Therefore
the number of operations for horizon of length $n$ is
\[
 O\bigg(n \binom{K+r-1}{r-1} (K+1) r\log(1+r)\bigg)
\]
for this part of the algorithm. In addition, the reward evaluation requires $O((K+r)r)$ operations. This estimate shows that the complexity grows rapidly with the number of grid intervals $K$. The required memory also increases rapidly with $K$,
so the algorithm cannot process big values of $K$.
In fact, $K$ and $r$ are the most important parameters that
determine running time.

\paragraph{Numerical results.}
The approximate optimal values $\hat{C}_n(0,\ldots,0)$
for  $r=1,2,3,4,5$ and $n=10,100,250, 500, 750, 1000, 2500$
are given in
Table~\ref{tab:GZ}. The values are obtained
with  number of grid intervals $K=4000, 998, 142, 56$ and $32$
for $r=1,2,3,4$ and $5$ respectively. For every $r> 1$ the  chosen value of $K$ is such that the number of states $|\cV_r|$ does not exceed 500,000.
This guarantees a reasonable running time of the code. Note
that the entries in the table  are approximations; in particular, the coarse grid for $r=5$
produces a non--monotone value at $n=750$.
\par
The numbers in the table for $r=1$
coincide with those presented  in Table~8 of \citeasnoun{GiMo}.
For $r\geq 2$ we are not aware of works where
the optimal values for Gusein--Zade's problem with full information
are reported. However, the results of Table~\ref{tab:GZ} can be compared with those for the no-information setting where only relative ranks of random variables $X_1, \ldots, X_n$ are observed.
In this setting with
$r=2$ \citeasnoun{GuZa} proves that $\lim_{n\to\infty} V^*_n(q_{\rm GZ}) \approx 0.574$. In addition, Table~1 in \citeasnoun{GMZ} presents optimal values for different
$r$ and $n$. In particular, for $n=1000$ and $r=2, 5$
the optimal values in the problem with no--information are
$0.574174$ and $0.86123$. These numbers should be compared with $0.8009$ and
$0.9758$ respectively in the full information case; see the row of Table~\ref{th:GZ}
corresponding to $n=1000$.
This demonstrates the virtue of observing the values of $X_1, \ldots, X_n$ in comparison with the case where only relative ranks of the observations can be observed.
It is also worth noting  that for $r>5$ the winning probability becomes very close
to one; for this reason the numerical results in Table~\ref{tab:GZ}
are reported for $r\leq 5$ only. 

\begin{table}
\label{tab:Q3-values}
\setlength{\tabcolsep}{3pt}
\begin{tabular}{||r|ccccc||}
\hline
$n$ & $r=1$ & $r=2$ & $r=3$ & $r=4$ & $r=5$ 
\\
\hline
10 & 0.6087  & 0.8398  & 0.9384 & 0.9788  & 0.9938  
\\
100 & 0.5829 & 0.8044  & 0.9058 & 0.9542  & 0.9777 
\\
250 & 0.5813 & 0.8021  & 0.9036 & 0.9525  & 0.9767 
\\
500 & 0.5807 & 0.8013  & 0.9029 & 0.9521  & 0.9760  
\\
750 & 0.5805 & 0.8010  & 0.9027 & 0.9521  & 0.9765  
\\
1000 & 0.5804 & 0.8009 & 0.9026 & 0.9517  & 0.9758  
\\
2500 & 0.5803 & 0.8007 & 0.9025 & 0.9513  & 0.9750 
\\
\hline
\end{tabular}
\centering
\caption{Approximate optimal values $\hat{C}_n(0)$ for Gusein-Zade's stopping problem for $n=10, 100, 250, 500, 750, 1000$ and $2500$ computed using grids with 4000, 998, 142, 56, and 32 intervals for $r=1,2,3,4,5$, respectively. Entries are numerical grid approximations, rounded to 4 decimal digits. \label{tab:GZ}}
\end{table}
\par

\section{Robbins' problem}\label{sec:Robbins}
In this setting $q(a)=n-a$, and, by \eqref{eq:g-t}, 
\begin{equation}\label{eq:g-Rob}
 g_t(k,x)= n-k- (n-t)(1-x)= t-k + (n-t)x,\;\;k=1, \ldots, t.
\end{equation}
According to 
\eqref{eq:adapted-reward}, \eqref{eq:Y}, and \eqref{eq:g-Rob}
\begin{align*}
 Y_t &=\bE\{A_t|\sF_t\} = g_t(R_t, X_t)= t-R_t+ (n-t)X_t.
\end{align*} 
\par 
Let $0_k$ denote the   
$k$--tuple of zeros, $0_k:=(\underbrace{0, \ldots, 0}_k)$.
Define 
\[
 Z_t:=(X_{1:t}, \ldots, X_{t:t}, 0_{n-t}
 ),\;\;t=1, \ldots, n,\;\;\;Z_0:=0_n, 
\]
and let 
\begin{align}\label{eq:cZ-t}
\cZ_t:= \Big\{ z\in \bR^n: 1>z_1 > z_2 > \cdots > z_t>0, \,z_{t+1}=z_{t+2}=\cdots=z_n=0\Big\},\; t=1, \ldots, n.
\end{align}
We put also $\cZ_0=\{Z_0\}$. 
For $z\in \cZ_t$ and $u\in [0,1]$ define the insertion map 
$T^{(t)}: \cZ_t\times [0,1] \to \cZ_{t+1}$, $t=1, \ldots, n-1$ by the following formula 
\begin{equation}\label{eq:T-t}
 T^{(t)}(z,u)=\left\{\begin{array}{cc}
                T^{(t)}_1 (z,u), & z_1 <u\leq z_{0},\\
                T^{(t)}_2 (z,u), & z_2 <u\leq z_{1},\\
                \vdots & \vdots \\
                T^{(t)}_{t+1}(z,u), & u < z_t,                
               \end{array}
               \right.
\end{equation}
where $z_0=1$, and  
\begin{align*}
 &T_j^{(t)}(z,u):= (z_1, \ldots, z_{j-1}, u, z_{j}, \ldots, z_t, 0_{n-t-1}),
 \;\;\;\;j=1, \ldots, t+1,
 \\
 &T^{(0)}(0_n, 
 u)
=(u, 0_n). 
 \end{align*}
 In words, $T^{(t)}(z,u)$ returns a vector of length $n$ whose first 
$t+1$ components contain ordered values of vector $(z,u)$, 
while the remaining 
$n-t+1$ components are zeros.
\subsection{Optimal stopping rule and optimal value}
The next statement characterizes the optimal stopping rule and the 
corresponding optimal 
value in  Robbins' problem.
\begin{theorem}\label{th:Robbins}
 Let $C_m(z)$  be functions on $\cZ_{n-m}$, $m=0, \ldots, n$, defined recursively
by
\begin{align}
C_0(z) &= 0,
\label{eq:initialization-C}
\\
 C_m(z) &=\sum_{j=1}^{n-m+1}
\int_{z_j}^{z_{j-1}}\max\Big\{g_{n-m+1}(j, u), 
C_{m-1}\big(T_j^{(n-m)}(z, u)\big)\Big\} \rd u,
\label{eq:Cm}
\end{align}
where $g_m(j, u)$ and $T_j^{(n-m)}(\cdot, \cdot)$ 
are given in \eqref{eq:g-Rob} and \eqref{eq:T-t} respectively. Then  
 the optimal stopping rule is 
\begin{equation}\label{eq:Rob-stop}
 \tau_*:=\inf \Big\{1\leq t\leq n: t-R_t+(n-t)X_t \geq C_{n-t}(Z_t)\Big\},
\end{equation}
and the corresponding optimal value of Robbins' problem is
\[
 W_n^*(q_{\rm R})= n-C_n(0).
\]
\end{theorem}
\begin{remark}
 Because $g_n(k, x)=n-k$ and $C_0(z)=0$, 
 the inequality in \eqref{eq:Rob-stop} is always fulfilled for $t=n$. 
Thus, the stopping set in the definition of $\tau_*$ is
non--empty.
 \end{remark}
\begin{remark}
In the extended form
 the recursion is written as follows
 \begin{align}
  &C_m(z_1, \ldots, z_{n-m}, 0_m)
  \nonumber
  \\
  &= \sum_{j=1}^{n-m+1}
  \int_{z_j}^{z_{j-1}} \max\Big\{n-j-(m-1)(1-u),\,
C_{m-1}\big(z_1,\ldots,z_{j-1}, u, z_j, z_{j+1},\ldots, z_{n-m}, 0_{m-1}\big)\Big\} \rd u.
\label{eq:C-Robbins}
\end{align}
\end{remark}

\begin{proof}
The proof goes along the same lines as the proof of Theorem~\ref{th:GZ}.
\par 
{\bf 1$^0$.} With the introduced  notation we have that  
$Z_{t+1}= T^{(t)}(Z_t, X_{t+1})$.
Consider the process 
\[
 S_t := (Z_t, X_t)= \big(T^{(t-1)}(Z_{t-1}, X_t), X_t\big),\;\;\;t=1,2, \ldots.
\]
This process $(S_t)_{t\geq 0}$ is a Markov chain on state  space 
$(\cS, \sB)$ with $\cS:=\cup_{t=1}^n \cS_t$ and 
$\cS_t:=\cZ_t\times [0,1]$:
 \[
  \bP\{ S_{t+1}\in B |S_t, S_{t-1}, \ldots, S_1\}=
  \bP\{ S_{t+1}\in B |S_t\}=
  \int_0^1 I\Big\{\big(T^{(t)}(Z_t, u), u\big)
  \in B\Big\}\rd u.
 \]
Indeed, 
 for any bounded measurable function $f$ we have 
 \[
  f(S_{t+1})= f\big(T^{(t)}(Z_t, X_{t+1}), X_{t+1}),
 \]
and since $X_{t+1}$ is independent of $\sF_t$,
\begin{align}\label{eq:transition-R}
 \bE \big\{f(S_{t+1})|\sF_t\}= \bE
 \big\{f\big(T^{(t)}(Z_t, X_{t+1}), X_{t+1})|\sF_t\}=
 \int_0^1 f\big(T^{(t)}(Z_t, u), u\big) \rd u.
\end{align}
Because the right hand side depends on $Z_t$ (and $S_t$) only, 
and $S_1, \ldots, S_t$ are $\sF_t$--measurable,
$(S_t)_{t\geq 1}$ is a Markov chain. We note that  $(z,x)\in \cS_t$ is a reachable state of  the Markov chain
$(S_t)_{t\geq 1}$ at time $t$ only if $x\in \{z_1,\ldots, z_t\}$. Thus,  
there are only $t$ states that can be reached by the Markov chain at time~$t$. 
\par
The adapted reward $Y_t$ is represented as a function  of the state 
of the Markov chain $(S_t)_{t\geq 0}$. 
For $(z,x)=(z_1, \ldots, z_{t},0_{n-t}, x)\in
\cZ_{t}\times [0,1]$ and 
$t=1, \ldots, n$ we define
\begin{align*}
 G_t(z, x) & := \left\{ \begin{array}{cl} 
 g_t(k, x), &\;\;\; x = z_k,\;k=1, \ldots,t,
 \\
 0, &\;\;\; \hbox{otherwise}.
 \end{array}
 \right.
 \end{align*}
 however, the 
above definition
of $G_t(z,x)$ can be extended for every $(z,x)\in \cZ_t\times [0,1]$.
 With this notation we have
\[
 Y_t= G_{t}(S_t),\;\;S_t := (Z_t, X_t)= \big(T^{(t-1)}(Z_{t-1}, X_t), X_t\big),\;\;\;t=1,2, \ldots.
\]
\par 
{\bf 2$^0$.}
Let $V_m(z, x)$ be the optimal expected reward from  state
$(z,x)\in \cS_{n-m}$ when we examine observation $X_t$ for stopping, and 
there are $m=n-t$ steps to go.
In view of the Markov property of $(S_t)$, 
the dynamic programming yields 
\begin{align*}
 &V_m(z,x)= 
 \max\big\{G_{n-m}(z,x),\; \bE \big[V_{m-1}(S_{t+1})|S_t=(z,x)\big]\big\},\;\;\;m=1, \ldots, n-1,
\\
&V_0(z, x) = G_n(z,x).
\end{align*}
Because the transition kernel in \eqref{eq:transition-R} does not depend on $x$, the conditional expectation in the above formula~is 
a function of $z$ 
only,   
\[
 C_m(z):=\bE \big[V_{m-1}(S_{t+1})|S_t=(z,x)\big] = \bE \big[V_{m-1}(S_{t+1})|Z_t=z\big],
\]
and the dynamic programming recursion is 
\begin{equation*}
 V_m(z,x)= 
 \max\big\{G_{n-m}(z,x),\; C_m(z)\big\},\;\;\;\forall m.
\end{equation*}
Furthermore, by definition of $T^{(t)}(z,u)$ in \eqref{eq:T-t}, 
\begin{align*}
C_m(z)  =
 \bE \big[V_{m-1}(S_{t+1})|Z_t=z\big]=
 \int_0^1 V_{m-1}\big(T^{(t)}(z, u), u\big) \rd u
 =
 \int_0^1 V_{m-1}\big(T^{(n-m)}(z, u), u\big) \rd u
\\
\;=\sum_{j=1}^{n-m+1}
\int_{z_j}^{z_{j-1}}\max\Big\{g_{n-m+1}(j, u), C_{m-1}
\big(T_j^{(n-m)}(z, u)\big)\Big\} \rd u,
\end{align*}
where the upper summation  limit $n-m+1$ is the number of insertion
ranks.
At the last observation when $m=0$ 
we put $C_0(z)=0$.
\par 
The optimal stopping rule is defined as follows. At time instance $t$
we have 
\begin{eqnarray*} 
Z_t &=& (X_{1:t},\ldots, X_{t:t}, 0_{n-t}),
             \;\;\;\;\;t=1,\ldots,n,
\end{eqnarray*}
and if $R_t=k$ then 
the reward from stopping at time $t$ is $g_{t}(k,X_t)$, while the reward 
from continuing is $C_{n-t}(Z_t)$. Therefore the optimal stopping rule
is 
\[
 \tau_*:=\inf \Big\{1\leq t\leq n: g_{t}(R_t,X_t)\geq C_{n-t}(Z_t)\Big\}.
\]
The optimal value of Robbins' problem is $W_n^*(q_{\rm R})= n- C_n(0)$.
The proof is completed.
\end{proof}
\subsection{Discussion}\label{sec:Discussion-Robbins}
It is instructive to compute functions $C_m(z)$ for $m=1$ and $m=2$.
Since $z_0=1$, it follows from \eqref{eq:C-Robbins} by straightforward algebra that
\begin{align}\label{eq:C-1}
 C_{1}(z_1, \ldots, z_{n-1}, 0)&= \sum_{j=1}^n \int_{z_j}^{z_{j-1}} \max\{n-j,0\}\rd u
 =\sum_{j=1}^{n-1} \int_{z_j}^{z_{j-1}}(n-j)\rd u= n- 1- \sum_{j=1}^{n-1} z_j.
\end{align}
Now we compute $C_2(z_1, \ldots, z_{n-2}, 0, 0)$. Denote for brevity
$Q:=\sum_{j=1}^{n-2} z_j$, and recall from \eqref{eq:g-Rob} that for $m=2$ we have
$g_{n-m+1}(j,u)=g_{n-1}(j, u)=n-1-j+u$.
In addition,
\[
 C_1(T_{j}^{(n-2)}(z,u))= n-1-Q-u,\;\;\;j=1, \ldots, n-2.
\]
Therefore, in view of
\eqref{eq:C-Robbins},
\begin{align*}
 C_2(z_1, \ldots, z_{n-2}, 0, 0) &=
 \sum_{j=1}^{n-1} \int_{z_j}^{z_{j-1}} \max\Big\{ n-1-j+u, n-1- Q -u\Big\}
 \rd u.
\end{align*}
Note that $\max\{a, b\}= b+ (a-b)_+$, where $(x)_+=\max\{x, 0\}$. Therefore
\[
 \max\Big\{ n-1-j+u, n-1-Q -u\Big\} = n-1-Q-u + (Q-j+2u)_+,
\]
and
\begin{align}
 &C_2(z_1,\ldots, z_{n-2}, 0, 0)
 \nonumber
 \\
 &\;\;=(n-1-Q)\sum_{j=1}^{n-1} (z_{j-1}-z_j)  - \sum_{j=1}^{n-1}\int_{z_j}^{z_{j-1}} u\rd u + \sum_{j=1}^{n-1}\int_{z_j}^{z_{j-1}} (Q-j+2u)_+\rd u
\nonumber
 \\
&\;\hspace{47mm}= n-\frac{3}{2}- Q  + \sum_{j=1}^{n-1} \int_{z_j}^{z_{j-1}} (Q-j+2u)_+\rd u
\nonumber
\\
&\;\hspace{47mm}= n-\frac{3}{2}- Q + \frac{1}{4}\sum_{j=1}^{n-1} \big[(Q-j+2z_{j-1})_+^2-(Q-j+2z_{j})_+^2\big],
\label{eq:C-2}
\end{align}
 where in the third line we have taken into account that $z_{0}=1$ and $z_{n-1}=0$. In particular, specializing this formula for  $n=2$ we obtain
 \[
  C_2(0,0)= 2- \frac{3}{2} +\frac{1}{4} \big[(0-1+2\cdot 1)_+^2 - (0-1+ 2\cdot 0)_+^2\big]=\frac{3}{4}.
 \]
Thus, the optimal value of Robbin's problem for $n=2$ is
$W_2^*(q_{\rm R}) = 2- C_2(0,0)=\frac{5}{4}$, which is the well known result.

\subsection{Numerical experiments}\label{sec:Numerics-Robbins}

\paragraph{Algorithm.}
 The computer implementation of recursion~\eqref{eq:C-Robbins}
presents serious computational challenges.
At each step of the dynamic programming procedure the one--dimensional
integrals
in~\eqref{eq:C-Robbins}
should be approximated on the grid with $K$ intervals.
Let $x_i=i/K$, $i=0, \ldots, K$ be the grid nodes. According to \eqref{eq:cZ-t},
at every time
$t\in \{1, \ldots, n\}$ {\em the grid state} $x(v)$
is the $t$--tuple of non-increasing grid nodes appended by
$n-t$ zeros, $x(v)=(x_{v_1}, \ldots, x_{v_t}, 0_{n-t})$,
where
\[
 v\in \cV_t:=\Big\{(v_1, \ldots, v_t)\in \bN^t: K+1\geq v_1\geq v_2\geq\cdots\geq v_t\geq 1\Big\}.
\]
The total number of grid states at time $t$ is
\[
 |\cV_t|= \binom{K+t}{t}, \;\;\;|\cV_0|=1,
\]
which is the number of multisets of size $t$ in $K+1$ grid indicies.
\par
Let $\hat{C}_m(x(v))$ be an approximation for $C_m(x(v))$, where
$v\in \cV_{n-m}$, and $m=n-t$.
The analytic expressions for $C_1(z)$ and $C_2(z)$ are given in \eqref{eq:C-1}
and \eqref{eq:C-2}. For $m\geq 3$ the approximations $\hat{C}_m(x(v))$ are
computed as follows.
Fix any $v\in \cV_t$, and let $x(v)$ be the grid
state associated with~$v$. The integrals in
\eqref{eq:Cm}--\eqref{eq:C-Robbins} take the form
\begin{align}
 &\int_{x_{v_j}}^{x_{v_{j-1}}}\max\big\{ g_{n-m+1}(j, u), C_{m-1}\big(T_j^{(t)}(x(v),u)\big) \big\}\rd u
\nonumber
 \\
&\;\;\;\;\;= \sum_{i=v_j}^{v_{j-1}-1} \int_{x_i}^{x_{i+1}} \max\big\{ g_{n-m+1}(j, u), C_{m-1}\big(T_j^{(t)}(x(v),u)\big) \big\}\rd u.
\label{eq:the-integral}
\end{align}
Recall that
 \[ T_j^{(t)}(x(v), u)= (x_{v_1},\ldots, x_{v_{j-1}}, u, x_{v_j}, \ldots, x_{v_t}, 0_{n-t-1}).
 \]
For $i\in \{0, \ldots, K\}$ and $v\in \cV_t$ define
$w_i(v):={\rm sort}_{\downarrow}(v_1, \ldots, v_t, i)\in \cV_{t+1}$,
where ${\rm sort}_{\downarrow}$ is the operator of sorting the indices in descending order. With this notation,
\[
 T_j^{(t)}(x(v), x_i)= x(w_i(v)),\;\;\;i= 0,\ldots, K.
\]
The value $C_{m-1}(T_j^{(t)}(x(v), u))$
on  interval $[x_i, x_{i+1}]$
is approximated by interpolating
the previously computed values at points $x(w_i(v))$ and $x(w_{i+1}(v))$, $\hat{C}_{m-1}(x(w_i(v)))$ and $\hat{C}_{m-1}(x(w_{i+1}(v)))$, i.e.,
\[
 \hat{C}_{m-1}\big(T_j^{(t)}(x(v), u)\big) = \frac{x_{i+1}-u}{x_{i+1}-x_i}\hat{C}_{m-1}(x(w_i(v)))+ \frac{u-x_{i}}{x_{i+1}-x_i}\hat{C}_{m-1}\big(x(w_{i+1}(v))\big).
\]
Then the integral in \eqref{eq:the-integral} is computed analytically. This is possible because $g_{n-m+1}(j, u)$ is an affine function in $u$; so
the integrand is a maximum of two affine functions. Thus, the errors in approximation of functions $C_m(z)$ are caused by interpolation at each step of the dynamic programming procedure.
\par
For each state $x(v)$ the algorithm reconstructs its coordinates in a long vector representation,
and for each interval of integration it finds the rank of the inserted observation. These operations become more expensive as $t$ grows. A conservative
upper bound on the number of operations is
\[
 O\bigg(\sum_{t=0}^{n-2} \binom{K+t}{t} \Big[t\log(K+1)+ K(t+1)\log(t+2)\Big]\bigg).
\]
The required memory also grows very fast with $K$ and $n$.
All this implies that the recursion \eqref{eq:C-Robbins} based on the whole memory
can be implemented with reasonable accuracy
only for small values of horizon $n$.

\paragraph{Numerical results.}

\begin{table}
\begin{center}
 \begin{tabular}{||c|c|c||c|c|c||} \hline
$n$ & $K$ & $\hat{W}_n^*$ & $n$ & $K$ & $\hat{W}_n^*$
\\ \hline\hline
1 & - & 1.0000 &  6& 250 & 1.6321
\\
2 & - &1.2500 & 7&200 & 1.6823
\\
3 & 10000 & 1.3916 & 8 & 100 & 1.7244
\\
4 & 2000 & 1.4933 & 9& 50 & 1.7606
\\
5 & 600 &1.5707 &  10 & 50 & 1.7924\\
\hline\hline
 \end{tabular}
 \caption{Approximate optimal values of $\hat{W}_n^*$ computed for
 $n=1, \ldots, 10$ with number of grid intervals $K$ given in the table. The numbers are rounded to 4 decimal digits. \label{tab:Robbins}}
 \end{center}
\end{table}
Approximations  $\hat{W}_n^*$ to the optimal values
$W_n^*(q_{\rm R})=n-C_n(0,\ldots, 0)$
for $n=1, \ldots, 10$ are reported in Table~\ref{tab:Robbins}.
The optimal values for $n=1, 2$ are included for completeness, while for $n\geq 3$
they are computed using $K$ grid intervals with the values of $K$ 
indicated in the table.
\par
The interpolation error at every step of the dynamic programming procedure is the main source of inaccuracy in computing the optimal values.
These errors accumulate with progression of the computational procedure. 
A conservative bound on the
accuracy of $\hat{W}_n^*$ can be estimated as follows.
In view of
\eqref{eq:C-Robbins}, at step $m$
the integrand satisfies Lipschitz's condition with constant $m-1$; so 
the
error of interpolation
at step $m$ of the procedure is of the order $O(m/K)$. Then 
the final approximation
error in computing $W_n^*$ is of the order $O(n^2/K)$.
Thus, theoretically, the number of grid intervals 
should grow with the horizon in order to maintain 
reasonable accuracy in evaluating the optimal values.
Unfortunately, this is impossible in view of the computer power limitations and memory constraints. 
The algorithm complexity grows very fast
with $K$ and $n$, and computation of results in Table~\ref{tab:Robbins} is time consuming. For example,
for $n=7$ and $K=200$ the computation took approximately 14340 seconds on
a standard desktop computer.
\par
The results in Table~\ref{tab:Robbins} can be compared with exact results
in  the literature obtained for $n=3, 4$.
\citeasnoun{Ass96} report the value $W_3^*\approx 1.3915$, while
\citeasnoun{Dendievel} show that
$W_4^*\approx 1.4932$.  Our approximate computation with $K=10000$
grid intervals yields
$\hat{W}_3^*\approx 1.391552086$, while for $n=4$ with $K=3000$ grid intervals
we have $\hat{W}_4^* \approx 1.493297613$.
These results are in line with the constants reported in the literature.
We are not aware of exact results for $n\geq 5$.
\par
Recently upper bounds on the optimal values $W_n^*$ for $5\leq n\leq 100$ were reported in \citeasnoun{Brice2026}. In this paper the Robbins' problem is modeled
as an infinite Markov decision process, and finite abstractions are constructed.
A comparison of the results in Table~3 of  \citeasnoun{Brice2026} with
Table~\ref{tab:Robbins} for $5\leq n\leq 10$ reveals that the numbers in \citeasnoun{Brice2026} are  bigger than those given in Table~\ref{tab:Robbins}. However, they coincide with our results in the first two decimal digits. It is worth noting that the results in Table~\ref{tab:Robbins}
do not provide upper bounds on the optimal values; they are obtained by approximation of the optimal recursion, and they are not associated
with a particular stopping strategy. Nevertheless, in principle,
based on the optimal  recursion \eqref{eq:C-Robbins},
both upper and lower bounds on the optimal values can be developed
by judicious  approximation of functions $C_{m}(z)$ at each step of the 
dynamic programming procedure.
\par
A natural way to reduce computational difficulties in implementation of the whole history dynamic program is to retain a fixed number, say, $r$ highest order statistics at every step of the procedure. This corresponds to restricting the
dimensionality of the state space of the corresponding Markov representation. Such a restriction can allow one to obtain approximate optimal values
for bigger values of~$n$. We tried to implement this idea in our numerical experiments; however, the numerical results were discouraging: for reasonable running times $r$ was chosen to be $3$ or $5$, and the obtained results were not in the range of optimal values given by
the memoryless threshold strategies. In general, further numerical work
is required in order to obtain more precise approximations
to the optimal values for bigger values of $n$.

\section{Concluding remarks}\label{sec:conclusion}
We close this paper with  concluding remarks.
\par\medskip
1.
The structure of optimal stopping rules established in Theorems~\ref{th:GZ} and~\ref{th:Robbins}
reveals some interesting features.  
Construction of the Markov chains $(S_t)_{t\geq 1}$ in  theorems' proofs 
depends very weakly on the reward function $q$: the support of $q$ determines the dimensionality of the state space. For example, in Gusein-Zade's problem $q$ is supported on $\{1, \ldots, r\}$, so that the state space of the embedded 
Markov chain has fixed dimensionality  $r+1$, which  does not increase with $n$.  In contrast, in Robbins' $n$--period stopping problem the dimensionality of the 
state space of the chain is taken to be $n+1$ to accommodate for the memory growth. 
The reward function $q$ appears in the Markov representation of random variables 
$(Y_t)$. It determines the reward from visiting a state by the chain $(S_t)_{t\geq 1}$. 

\par\medskip  
2. The structure of the optimal stopping rules described in 
 item~1 above allows one to
develop  optimal stopping rules for
a broad range of other stopping problems with full information and general rank--dependent reward function $q$.
For example, consider the problem of selecting the $r$th best observation
with full information, i.e., we want to find a stopping rule
$\tau_*\in\sT$ such that 
$\bP\{A_{\tau_*}=r\}=\max_{\tau\in \sT}\bP\{A_\tau=r\}$.
The no--information variant of this problem was studied in the literature.
The specific case $r=2$ is the so--called postdoc problem that was
stated and solved by \citeasnoun{Vanderbei2012}; see  \citeasnoun{GMZ} for further references. We are not aware of results on the full information variant of this problem. The construction used in the proof of Theorem~\ref{th:GZ} enables us to fully characterize the optimal solution. In particular,
here $q(a)=I(a=r)$, and, in view of  \eqref{eq:g-t},
\begin{equation}\label{eq:postdoc}
 g_t(k, x)= I(k\leq r) \binom{n-t}{r-k} (1-x)^{r-k} x^{n-t-r+k},\;\;\;k=1, \ldots, n.
\end{equation}
The Markov chain $(S_t)_{t\geq 1}$ underlying the 
representation of adapted rewards $Y_t$
remains the same. The only change is in
functions $g_t$; now they are  given by \eqref{eq:postdoc}.
Thus, the optimal stopping rule is again
given by Theorem~\ref{th:GZ} with the above new definition of
$g_t(R_t, X_t)$.
The same reasoning is applicable  for arbitrary reward functions $q$, and for
variants of the full information stopping problems with random horizon.
\par\medskip 
3.  The form of the optimal strategy for Robbins' problem derived
in Theorem~\ref{th:Robbins} opens up new avenues for
investigating the limit of optimal values $W_n^*$ and
for developing $\epsilon$--optimal strategies based on a fixed number of observations.
In particular, it seems plausible that 
the existing bounds on the optimal value, which  are based on the 
memoryless threshold rules, 
can be improved by considering strategies of the type
\[
 \tau=\inf \big\{ 1\leq t\leq n: t- R_t + (n-t)X_t \geq c_{n-t}\big\},
\]
where $(c_t)$ are constants to be specified.

\subsubsection*{Acknowlegment}
We were assisted by the OpenAI Codex coding agent when
implementing the algorithm and writing the MATLAB code for the
approximate computation of the functions $C_m(z)$, which are given by
recursive formulas of Theorems~\ref{th:GZ} and \ref{th:Robbins}. The numerical results
reported in Tables~\ref{tab:GZ} and \ref{tab:Robbins} are based on
this code; it is available from the authors upon request.  All of the theoretical results,
proofs and text are our own.
The research is supported by the BSF grant 2020063.

\bibliographystyle{agsm}

\end{document}